\documentclass[a4paper, 12pt]{amsart}
\usepackage{amsmath}
\usepackage{amssymb, latexsym, amsthm}
\usepackage[all]{xy}

\newtheorem{Thm}{Theorem}[section]
\newtheorem{Prop}[Thm]{Proposition}
\newtheorem{Cor}[Thm]{Corollary}
\newtheorem{Lem}[Thm]{Lemma}

\theoremstyle{definition}

\numberwithin{equation}{section}

\begin{document}

\newcommand{\Coim}{\mathrm{Coim}}
\newcommand{\Z}[0]{\mathbb{Z}}
\newcommand{\Q}[0]{\mathbb{Q}}
\newcommand{\F}[0]{\mathbb{F}}
\newcommand{\C}[0]{\mathbb{C}}
\newcommand{\N}[0]{\mathbb{N}}
\renewcommand{\O}[0]{\mathcal{O}}
\newcommand{\p}[0]{\mathfrak{p}}
\newcommand{\m}[0]{\mathrm{m}}
\newcommand{\Tr}{\mathrm{Tr}}
\newcommand{\Hom}[0]{\mathrm{Hom}}
\newcommand{\Gal}[0]{\mathrm{Gal}}
\newcommand{\Res}[0]{\mathrm{Res}}
\newcommand{\id}{\mathrm{id}}
\newcommand{\cl}{\mathrm{cl}}
\newcommand{\mult}{\mathrm{mult}}
\newcommand{\adm}{\mathrm{adm}}
\newcommand{\tr}{\mathrm{tr}}
\newcommand{\pr}{\mathrm{pr}}
\newcommand{\Ker}{\mathrm{Ker}}
\newcommand{\ab}{\mathrm{ab}}
\newcommand{\sep}{\mathrm{sep}}
\newcommand{\triv}{\mathrm{triv}}
\newcommand{\alg}{\mathrm{alg}}
\newcommand{\ur}{\mathrm{ur}}
\newcommand{\Coker}{\mathrm{Coker}}
\newcommand{\Aut}{\mathrm{Aut}}
\newcommand{\Ext}{\mathrm{Ext}}
\newcommand{\Iso}{\mathrm{Iso}}
\newcommand{\M}{\mathcal{M}}
\newcommand{\GL}{\mathrm{GL}}
\newcommand{\Fil}{\mathrm{Fil}}
\newcommand{\an}{\mathrm{an}}
\renewcommand{\c}{\mathcal }
\newcommand{\W}{\mathcal W}
\newcommand{\R}{\mathcal R}
\newcommand{\crys}{\mathrm{crys}}
\newcommand{\st}{\mathrm{st}}
\newcommand{\CM}{\mathrm{CM\Gamma }}
\newcommand{\CV}{\mathcal{C}\mathcal{V}}
\newcommand{\G}{\mathrm{G}}
\newcommand{\Map}{\mathrm{Map}}
\newcommand{\Sym}{\mathrm{Sym}}
\newcommand{\Spec}{\mathrm{Spec}}
\newcommand{\Gr}{\mathrm{Gr}}
\newcommand{\I}{\mathrm{Im}}
\newcommand{\Frac}{\mathrm{Frac}}
\newcommand{\LT}{\mathrm{LT}}
\newcommand{\Alg}{\mathrm{Alg}}
\newcommand{\MG}{\mathrm{M\Gamma }}
\newcommand{\To}{\longrightarrow}
\newcommand{\md}{\mathrm{mod}}
\newcommand{\MF}{\mathrm{MF}}
\newcommand{\CMF}{\mathcal{M}\mathcal{F}}
\newcommand{\Aug}{\mathrm{Aug}}
\renewcommand{\c}{\mathcal }
\newcommand{\uL}{\underline{\mathcal L}}
\newcommand{\Md}{\mathrm{Md}}
\newcommand{\wt}{\widetilde}
\newcommand{\op}{\mathrm}
\newcommand {\w}{\op{wt}}
\newcommand{\Ad}{\op{Ad}}
\newcommand{\ad}{\op{ad}}
\newcommand{\D}{\mathcal D}

\title{Ramification estimate for Fontaine-Laffaille Galois modules}
\author{Victor Abrashkin}
\address{Department of Mathematical Sciences, 
Durham University, Science Laboratories, 
South Rd, Durham DH1 3LE, United Kingdom \ \&\ Steklov 
Institute, Gubkina str. 8, 119991, Moscow, Russia
}
\email{victor.abrashkin@durham.ac.uk}
\date{}
\subjclass[2010]{}

\begin{abstract} 
Suppose $K$ is unramified over $\Q _p$ and $\Gamma _K=\op{Gal}(\bar K/K)$. 
Let $H$ be a torsion $\Gamma _K$-equivariant subquotient of crystalline $\Q _p[\Gamma _K]$-module 
with HT weights from $[0,p-2]$. We give a new proof of Fontaine's conjecture about the 
triviality of action of some ramification subgroups $\Gamma _K^{(v)}$ 
on $H$. The earlier author's  
proof from [1] contains a gap and proves this conjecture only for some subgroups of 
index $p$ in $\Gamma _K^{(v)}$. 
\end{abstract}
\maketitle

\section*{Introduction} \label{S0} 

Let $W(k)$ be the ring of Witt vectors with coefficients in 
a perfect field $k$ of characteristic $p$. 
Consider the field $K=W(k)[1/p]$, choose its algebraic closure $\bar K$ and 
set $\Gamma _K=\op{Gal}(\bar K/K)$. Denote by $\C _p$ the completion of $\bar K$ and 
use the notation $O_{\C _p}$ for its valuation ring. 

For $a\in\Z _{\geqslant 0}$, let $\op{M\Gamma }^{cr}_{\Q _p}(a)$ be the category of crystalline 
$\Q _p[\Gamma _K]$-modules with Hodge-Tate weights from $[0,a]$. 
Define the full subcategory $\op{M\Gamma }_N^{cr}(a)$ of the category of 
$\Gamma _K$-modules consisting of 
$H=H_1/H_2$, where $H_1,H_2$ are $\Gamma _K$-invariant lattices 
in $V\in\op{M\Gamma }_{\Q _p}^{cr}(a)$ and $p^NH_1\subset H_2\subset H_1$. 
J.-M.\,Fontaine conjectured in \cite{Fo} that the ramification subgroups 
$\Gamma _K^{(v)}$ act on $H\in\op{M\Gamma }_N^{cr}(a)$ trivially if 
$v>N-1+a/(p-1)$. The author suggested in \cite{Ab1} a proof of this conjecture 
under the assumption $0\leqslant a\leqslant p-2$. 

It was pointed 
recently by Sh. Hattori to the author that the proof in \cite{Ab1} has a gap. 
More precisely, let $R=\varprojlim (O_{\C _p}/p)$ be Fontaine's ring. 
For  
$r=(o_n\,\op{mod}\,p)_{n\geqslant 0}\in R$ and $m\in\Z$, set 
$r^{(m)}=\underset{n\to\infty }\lim o_n^{p^{n+m}}\in O_{\C _p}$ and consider   
Fontaine's map $\gamma :W_N(R)\To O_{\C _p}/p^N$, where     
$(r_0,\dots ,r_{N-1})\mapsto\sum _{0\leqslant i<N}p^ir_i^{(i)}\,\op{mod}\,p^N$.  
Consider the projection \linebreak 
$(\bar o_0,\dots ,\bar o_N,\dots )\mapsto \bar o_N$ from 
$R$ to $O_{\C _p}/p$  
and denote the image of $\Ker\,\gamma $ in $W_N(O_{\C _p}/p)$ 
by $W_N^1(O_{\C _p}/p)$. 
This is principal ideal and in order to apply 
Fontaine's criterion about the triviality of the action of ramification subgroups 
from \cite{Fo}, we needed an element  
of $W_N(L)$, where $L$ is a finite extension of $K$ with ``small'' ramification, 
which generates $W^1_N(O_{\C _p}/p)$. Our ``truncation'' argument in \cite{Ab1} does 
not actually work: the resulting element does not belong to $W^1_N(O_{\C _p}/p)$. 
In the moment the author is inclined to believe that such 
an element does not exist if $N>1$. 
Nevertheless, our proof in \cite{Ab1} gives the Fontaine conjecture up to index $p$:  
the groups $\Gamma _K^{(v)}$ just should be 
replaced by the groups $\Gamma _K^{(v)}\cap\Gamma _{K(\zeta _{N+1})}$, 
where $\zeta _{N+1}$ is a primitive 
$p^{N+1}$-th root of unity. 

The above difficulty appears in many other situations when we try to escape from 
``$R$-constructions'' (e.g. $W(R)$, $A_{cr}$, etc) to $p$-adic 
constructions inside $\C _p$. In this paper we prove Fontaine's conjecture by 
applying methods from \cite{Ab2}. These methods were used earlier to study 
ramification properties in the characteristic $p$ case. As a matter of fact, this is the 
first time when we use them in the mixed characteristic situation. 
\medskip

\section{Construction of torsion crystalline representations} \label{S1}

The ring $R$ is perfect of characteristic $p$, it is 
provided with the valution $v_R$ such that $v_R(r):=
\underset{n\to\infty}\lim p^nv_p(o_n)$, where $r=(o_n\,\op{mod}\,p)_{n\geqslant 0}$.  
With respect to $v_R$, $R$ is complete and the field 
$R_0:=\op{Frac}R$ is algebraically closed. Note that $R$ and $R_0$ 
are provided with natural $\Gamma _K$-action. Denote by $\sigma $ 
the Frobenius endomorphism of  
$R$ and $W(R)$ and by $\m _R$ the maximal ideal of $R$.  

\subsection{} 
Let $\c G=\op{Spf}\,W(k)[[X]]$ be the Lubin-Tate 1-dimensional formal group 
over $W(k)$ such that $p\op{id}_{\c G}(X)=pX+X^p$. Then $\op{End}_{W(k)}\c G=\Z _p$ and for any 
$l\in\Z _p$, $(l\id _{\c G})(X)\equiv lX\,\op{mod}\, X^p$. 

Fix $N\in\N$. 

For $i\geqslant 0$, choose 
$o_i\in O_{\C _p}$ such that $o_0=0$, $o_1\ne 0$ and $p\op{id}_{\c G}(o_{i+1})=o_i$. 
Set $\tilde u=(o_{N+i}\,\op{mod}\,p)_{i\geqslant 0}\in R$. 
Then $\c K:=k((\tilde u))$ is a complete discrete valuation 
closed subfield in $R_0$. If $\c K_{sep}$ is the separable closure of $\c K$ 
in $R_0$ then $\c K_{sep}$ is separably closed and its completion 
coincides with $R_0$. The theory of the field-of-norms functor \cite{Wtb1} 
identifies $\Gamma _{\c K}$ with a closed subgroup in $\Gamma _K$. 
The quotient $\Gamma _K/\Gamma _{\c K}$ acts strictly on $\c K$. 
More precisely, there is a group epimorphism 
$\kappa :\Gamma _K\To\op{Aut}_{W(k)}\c G\simeq\Z _p^*$ such that 
if $g\in\Gamma _K$ then $\kappa (g)\in\Z _p[[X]]$ and 
$\kappa (g)(X)\equiv \chi (g)X\,\op{mod}\,X^p$ with $\chi (g)\in\Z _p^*$. 
(Actually, $g\mapsto\chi (g)$ is the cyclotomic character.) With this notation we have 
$g(\tilde u)=\kappa (g)(\tilde u)$.

Use the $p$-basis $\{\tilde u\}$ for separable extensions $\c E$ of $\c K$ in $\c K_{sep}$ 
to construct the system of lifts $O_N(\c E)$ of $\c E$ modulo $p^N$. 
Recall that 
$O_N(\c E)=W_N(\sigma ^{N-1}\c E)[u_N]\subset W_N(\c E)$ 
and $O_N(\c K)=W_N(k)((u_N))$, 
where $u_N$ is the Teichmuller representative of $\tilde u$ in $W_N(\c K)$. 
This construction essentially depends on a choice of $p$-basis in $\c K$. 
If, say, $\{u'\}$ is another $p$-basis for $\c K$ and $O'_N(\c E)$ are the appropriate lifts 
then $O_N(\c E)$ and $O'_N(\c E)$ are not very much different one from another: they can be related 
by the natural embeddings 
$\sigma ^{N-1}O_N(\c E)\subset W(\sigma ^{N-1}\c E)\subset O'_N(\c E)$.  
The lifts $O_N(\c E)$ are provided with the endomorphism $\sigma $ 
such that $\sigma u_N=u_N^p$, and $O_N(\c K_{sep})$ is provided with 
continuous $\Gamma _{\c K}$-action.  
 
If $\tau $ is a continuous automorphism of $\c E$ then generally $\tau $ can't be lifted 
to an automorphism of $O_N(\c E)$ (but it can always be lifted to $W_N(\c E)$).  
In many cases it is sufficient to 
use ``the lift'' 
$\hat\tau :\sigma ^{N-1}O_N(\c E)\To O_N(\c E)$ induced by 
$W_N(\tau ):W_N(\c E)\To W_N(\c E)$. In other words, $\hat\tau $ 
is defined only on a part of $O_N(\c E)$, but   
$\hat\tau\,\op{mod}\,p=\sigma ^{N-1}\circ\tau :\sigma ^{N-1}\c E\To\sigma ^{N-1}\c E$  
and, therefore, $\tau $ can be uniquely recovered from the ``lift'' $\hat\tau $. 

On the other hand, any continuous automorphism $\tau $ of $\c K=k((\tilde u))$ 
can be lifted to an automorphism $\tau ^{(N)}$ of $O_N(\c K)=W_N(k)((u_N))$ (use that  
$u_N\,\op{mod}\,p=\tilde u$). Taking into account the existence of a lift $\tau _{sep}$ 
of $\tau $ to $\c K_{sep}$ we obtain a lift $\tau ^{(N)}_{sep}$ of $\tau $ to 
$O_N(\c K_{sep})=W_N(\sigma ^{N-1}\c K_{sep})[u_N]$.

Set $O_N^0:=O_N(\c K_{sep})\cap W_N(O_{sep})$ and 
$O_N^+:=O_N(\c K_{sep})\cap W_N(\m _{sep})$, where 
$\m _{sep}$ is the maximal ideal of the valuation ring $O_{sep}$ of $\c K_{sep}$. 
Then $\sigma (O_N^0)\subset O_N^0$, 
$\sigma (O_N^+)\subset O_N^+$ and $\bigcap _{n\geqslant 0}\sigma ^n(O_N^+)=0$. Note that 
$O^0_N(\c K):=O^0_N\cap O_N(\c K)=W_N(k)[[u_N]]$, 
$O^+_N(\c K):=O^+_N\cap O_N(\c K)=u_NW_N(k)[[u_N]]$ and 
$O_N(\c K)=O_N^0(\c K)[u_N^{-1}]=W_N(k)((u_N))$.  

For $0\leqslant m\leqslant N$, introduce   
$$u_m=(p^{N-m}\id _{\c G})(u_N)\in O_N^0(\c K)$$ 
Then $u_0=\sigma u_1=pu_1+u_1^p$, $t=u_0/u_1=p+u_1^{p-1}\in O_N^0(\c K)$ 
and $u_0^{p-1}=t^p-pt^{p-1}$.  As a matter of fact, $u_0, u_1, t$ depend only on $\tilde u$. 
Indeed, if $u'\in W_N(R)$ and $u'\,\op{mod}\,pW_N(R)=\tilde u$ then in $O_N(\c K)$ we have 
$u_1=(p^{N-1}\id _{\c G})(u')$. 

\begin{Lem} \label{L1.1} 
 Suppose $g\in\Gamma _K$. Then 
\medskip 

{\rm a)} $g(u_0)\equiv \chi (g)u_0\,\op{mod}\,u_0^pO_N^0(\c K)$;
\medskip 

{\rm b)} $\sigma (g(t)/t)\equiv 1\,\op{mod}\,u_0^{p-1}O_N^0(\c K)$.
\end{Lem}

\begin{proof} 
$g(u_1)=(p^{N-1}\id _{\c G})(g(u_N))=\kappa (g)(u_1)
\equiv \chi (g)u_1\,\op{mod}\,u_1^pO_N^0(\c K)$ 
implies a) because $\sigma (u_1)=u_0$. Then  
$g(t)/t\equiv 1\,\op{mod}\,u_1^{p-1}O_N^0(\c K)$ and applying 
$\sigma $ we obtain b). 
 \end{proof}

\subsection{} 

Let $\c M\c F$ be the category of $W(k)$-modules $M$ provided with decreasing filtration 
by $W(k)$-submodules $M=M^0\supset\dots\supset  M^{p-1}\supset M^p=0$ and 
$\sigma $-linear morphisms $\varphi _i:M^i\To M$ such that for all $i$, 
$\varphi _i|_{M^{i+1}}=p\varphi _{i+1}$.

For $0\leqslant a\leqslant p-2$, introduce the filtered module $\c S_a $ such that 
\medskip 

--- $\c S_a=O_N^0/u_0^aO_N^+$;
\medskip 

--- for $0\leqslant i\leqslant a$, 
$\op{Fil}^i\c S_a=t^i\c S_a$;
\medskip 

--- $\varphi _i:\op{Fil}^i\c S_a\To \c S_a$ is $\sigma $-linear morphism  
such that $\varphi _i(t^i)=1$. 
\medskip 

Clearly, $\c S_a\in\c M\c F$ (use that $\sigma t\equiv p\,\op{mod}\,u_0^{p-1}$). In addition, 
Lemma \ref{L1.1} implies also that the action of $\Gamma _K$ 
preserves the structure of an object of the category $\c M\c F$ on $\c S_a$.

For $0\leqslant a<p$, define the category of 
filtered Fontaine-Laffaille modules $\op{MF}_N(a)$ as the full subcategory 
in $\c M\c F$ consisting of modules $M$ of finite length over $W_N(k)$ 
such that 
$M^{a+1}=0$ and $\sum \op{Im}\varphi _i=M$. 
We can assume that $M$ is given together with a functorial 
splitting of its filtration, i.e. there are submodules $N_i$ in $M$ such that 
for all $i$, $M^i=N_i\oplus M^{i+1}$.

Let $M\in\op{MF}_N(a)$ and   
$\wt{U}_a(M)=\Hom _{\c M\c F}(M,\c S_a)$. Then the correspondence 
$M\mapsto \wt{U}_a(M)$ determines the functor $\wt{U}_a$ from $\op{MF}_N(a)$ to the category 
of $\Gamma _K$-modules.

\begin{Prop} \label{P1.2}
 If $0\leqslant a\leqslant p-2$ and $H\in\op{M\Gamma }_N^{cr}(a)$ then there is 
$M\in\op{MF}_N(a)$ such that $\wt{U}_a(M)=H$. 
\end{Prop}

\begin{proof} Recall briefly the main ingredients of the 
Fontaine-Laffaille theory \cite{FL}.  
The $p^N$-torsion crystalline ring 
$A_{cr.N}:=A_{cr}/p^N$ appears as  
the divided power envelope 
of $W_N(R)$ with respect to $\Ker\,\gamma $.  
We need the following construction of a generator of $\Ker\,\gamma $. 
(Note that we have a natural inclusion of $W(k)$-modules 
$O_N^0\subset W_N(R)$.) 

\begin{Lem} \label{L1.3}
 $\Ker\,\gamma =tW_N(R)$. 
\end{Lem}

\begin{proof} We have $\gamma (u_N)\equiv 
o_N\,\op{mod}\,pO_{\C _p}$, therefore, 
$\gamma (u_0)\equiv 0\,\op{mod}\,p^{N}O_{\C _p}$ and 
$t\in\Ker\gamma $. 
On the other hand, $t\equiv u_1^{p-1}\equiv [r]\,\op{mod}\,pW(R)$, where 
$r\in R$ is such that $r^{(0)}\equiv o_1^{p-1}\equiv -p\,\op{mod}\,p^{p/(p-1)}O_{\C _p}$. 
Therefore, $v_p(r^{(0)})=1$ and $t$  generates $\Ker\gamma $, cf. \cite{FL}. 
\end{proof}

By above Lemma, 
$A_{cr,N}=W_N(R)[\{\gamma _i(t)\ |\ i\geqslant 1\}]$, where 
$\gamma _i(t)$ are the $i$-th divided powers of $t$. 
Then  the identity $\gamma _p(t)=t^{p-1}+u_0^{p-1}/p$ implies that 
$A_{cr, N}=W_N(R)[\{\gamma _i(u_0^{p-1}/p)\ |\ i\geqslant 1\}]$. 

Recall that $A_{cr,N}\in\c M\c F$ with:
\medskip  

--- the filtration $\op{Fil}^iA_{cr,N}$, $0\leqslant i<p$, 
 generated as ideal 
by $t^i$ and all $\gamma _j(u_0^{p-1}/p)$, $j\geqslant 1$;
\medskip 

--- the $\sigma $-linear morphisms $\varphi _i:\op{Fil}^iA_{cr,N}\To A_{cr,N}$ 
(which come from $\sigma /p^i$ on $A_{cr}$) such that $\varphi _i(t^i)=(1+u_0^{p-1}/p)^i$ and 
$\varphi _i(u_0^{p-1}/p)=p^{p-1-i}(u_0^{p-1}/p)(1+u_0^{p-1}/p)^{p-1}$. 
\medskip 

Then the Fontaine-Laffaille functor $U_a$ attaches to $M\in\op{MF}_N(a)$ the 
$\Gamma _K$-module $\Hom _{\c M\c F}(M,A_{cr,N})$. 
This functor is fully-faithful (we assume that $a\leqslant p-2$) and, therefore, 
there is $M\in\op{MF}_N(a)$ such that $U_a(M)=H$.

Consider the $W(k)$-module $\c W^a_{N}=W_N(R)/u_0^aW_N(\m _R)$ with 
the filtration induced by the filtration $W^i_N(R)=t^iW_N(R)$ and 
$\sigma $-linear morphisms $\varphi _i$ such that $\varphi _i(t^i)=1$. 
 Prove that we have an identification of 
$\Gamma _K$-modules  $H=\Hom _{\c M\c F}(M,\c W_N^a)$. 

Indeed, let 
$T_a$ be the maximal element  in the family of 
all ideals $I$ of $A_{cr,N}$ such that $\varphi _a$ induces a 
nilpotent endomorphism of $I$. Then 
for any $M\in\op{MF}_N(a)$, 
$U_a(M)=\Hom _{\c M\c F}(M,A_{cr,N}/T_a)$. 
By straightforward calculations we can see that 
$T_a$ is generated by the elements of 
$u_0^aW_N(\m _R)$ and all $\gamma _j(u_0^{p-1}/p)$, $j\geqslant 1$. 
It remains to note that we have a natural identification  
$A_{cr,N}/T_a=\c W^a_{N}$ in the category $\c M\c F$. 

Consider the natural embedding $O_N^0\To W_N(R)$ and the induced natural 
map $\iota _a:  
\c S_a\To \c W_N^a$ in $\c M\c F$. Prove that $\iota _{a*}:  
\wt{U}_a(M)\to H$ is isomorphism of $\Gamma _K$-modules.   

Choose 
$W(k)$-submodules $N_i$ in $M^i$ such that $M^i=N_i\oplus M^{i+1}$ and choose 
vectors $\bar n_i$ whose coordinates give a minimal system of 
generators of $N_i$. Then the structure of $M$ can be given by 
the matrix relation 
$(\varphi _a(\bar n_a), \dots ,\varphi _0(\bar n_0))=
(\bar n_a,\dots ,\bar n_0)C$, 
where $C$ is an invertible matrix with coefficients in $W(k)$. The elements of 
$H$ are identified with the residues $(\bar u_a,\dots ,\bar u_0)
\,\op{mod}\,u_0^aW_N(\m _{\sep})$ 
where the vectors $(\bar u_a,\dots ,\bar u_0)$ have coefficients in $W_N(\c K_{sep})$ and 
satisfy the following system of equations  (use that $\varphi _a$ is topologically 
nilpotent on $u_N^aW_N(\m _{sep})$)
$$\left (\frac{\sigma \bar u_a}{\sigma t^a}, \dots , 
\frac{\sigma\bar u_i}{\sigma t^i}, \dots ,\sigma (\bar u_0)\right )=
(\bar u_a, \dots ,\bar u_0)C\, 
$$

In particular, if $\bar u=(\bar u_a, \dots , \bar u_0)$ then there is an invertible 
matrix $D$ with coefficients in $O_N(\c K)$ such that 
\begin{equation}\label{E1.1} 
\sigma (\bar u)D=\bar u\,. 
\end{equation}
 We know that all coordinates 
of $\sigma ^{N-1}\bar u$ belong to $\sigma ^{N-1}W_N(\c K_{sep})
\subset O_N(\c K_{sep})$. Then \eqref{E1.1} implies step-by-step that the vectors 
$\sigma ^{N-2}\bar u, \dots ,\bar u$  
have coordinates in $O_N(\c K_{sep})$.  It remains to note that 
$O_N^0=O_N(\c K_{sep})\cap W_N(O_{sep})$ and $O_N^+=O_N(\c K_{sep})\cap W_N(\m _{sep})$. 
The proposition is proved. 
\end{proof}

\section{Reformulation of the Fontaine conjecture} \label{S2}

\subsection{Review of ramification theory} \label{S2.1} 
Let $\c I_{\c K}$ be the group of all continuous automorphisms 
of $\c K_{sep}$ which keep invariant the residue field of $\c K_{sep}$ 
and preserve the extension of the normalised valuation $v_{\c K}$ 
of $\c K$ 
to $\c K_{sep}$. This group has 
a decreasing filtration by its ramification subgroups $\c I^{(v)}_{\c K}$ in 
upper numbering $v\geqslant 0$. Recall basic ingredients of the definition of 
this filtration following the papers \cite{De, Wtb1, Wtb2}. 

For any field extension $\c E$ of $\c K$ in $\c K_{sep}$, set    
$\c E_{sep}=\c K_{sep}$, in particular, $\c I_{\c E}=\c I_{\c K}$. All  
elements of $\c I_{\c K}$ preserve the extension $v_{\c E}$ of 
the normalised valuation on $\c E$ to $\c K_{sep}$.

For $x\geqslant 0$, set 
$\c I_{\c E,x}=\{\iota\in\c I_{\c E}\ |\ v_{\c E}(\iota (a)-a)
\geqslant 1+x\ \ \forall a\in\m _{\c E}\}$, where 
$\m _{\c E}$ is the maximal ideal in $O_{\c E}$.  

Denote by $\c I_{\c E/\c K}$ the set of all continuous embeddings of $\c E$ into 
$\c K_{sep}$ which induce the identity map on $\c K$ and the residue field 
$k_{\c E}$ of $\c E$. For $x\geqslant 0$, set 
$\c I_{\c E/\c K,x}=\c I_{\c E,x}\bigcap\c I_{\c E/\c K}\,$.  

If $\iota _1,\iota _2\in \c I_{\c E/\c K}$ and $x\geqslant 0$ 
then $\iota _1$ and $\iota _2$ are {\it $x$-equivalent} iff 
for any $a\in \m_\c E$,  $v_{\c E}(\iota _1(a)-\iota _2(a))\geqslant 1+x$. 
Denote by $(\c I_{\c E/\c K}:\c I_{\c E/\c K,x})$ 
the number of 
$x$-equivalent classes in $\c I_{\c E/\c K}$.  
Then the Herbrand function $\varphi _{\c E/\c K}$ 
can be defined for all $x\geqslant 0$, 
as 
$$\varphi _{\c E/\c K}(x)=\int _0^x(\c I_{\c E/\c K}:\c I_{\c E/\c K,x})^{-1}dx\,.$$ 
This function has the following properties:
\medskip 

$\bullet $\ $\varphi _{\c E/\c K}$ is a piece-wise linear function   
with finitely many edges;
\medskip 

$\bullet $\ if $\c K\subset \c E\subset \c H$ is a tower of finite field extensions 
in $\c K_{sep}$ then 
for any $x\geqslant 0$, 
$\varphi _{\c H/\c K}(x)=\varphi _{\c E/\c K}(\varphi _{\c H/\c E}(x))$. 
\medskip 

The ramification filtration 
$\{\c I_{\c K}^{(v)}\}_{v\geqslant 0}$ appears now as a decreasing sequence of   
the subgroups $\c I^{(v)}_{\c K}$ of $\c I_{\c K}$, where $\c I_{\c K}^{(v)}$ consists    
 of $\iota\in\c I_{\c K}$ such that for any finite extension 
$\c E$ of $\c K$, $\iota\in\c I_{\c E,v_{\c E}}$ with $\varphi _{\c E/\c K}(v_{\c E})=v$. 

If we replace the lower indices $\c K$ to $\c E$, the ramification 
filtration $\{\c I_{\c K}^{(v)}\}_{v\geqslant 0}$ 
is not changed as a whole, just only individual subgroups change their upper indices, 
that is $\c I_{\c K}^{(v)}=\c I_{\c E}^{(v_{\c E})}$.
 
Note 
that the inertia subgroup 
$\Gamma _{\c E}^0$ of $\Gamma _{\c E}=\op{Gal} (\c K_{sep}/{\c E})$ is a subgroup in $\c I_{\c E}$ 
and for any $v\geqslant 0$, 
the appropriate subgroup $\Gamma _{\c E}^{(v)}=\Gamma _{\c E}\cap\c I_{\c E}^{(v)}$ 
is just the ramification subgroup of $\Gamma _{\c E}$ with the upper number $v$ 
from \cite{Se1}.

\subsection{Statement of the main theorem} \label{S2.2} 

The main idea of our approach to the $\Gamma _K$-modules $\wt{U}_a(M)$ 
is related to the following fact. The filtered module $\c S_a$ 
depends only on the field $\c K$ and its uniformizer $\tilde u$. Therefore, 
$\c S_a$ can be identified with its analogue $\c S'_a$ constructed 
for any ramified extension $\c K'$ of $\c K$ together with its uniformizer 
$\tilde u'$. The whole 
group $\c I_{\c K}$ does not preserve the structure of $\c S_a$ 
but the ramification subgroups $\c I_{\c K}^{(v)}$, where $a>a^*_N:=(a+1)p^{N-1}-1$ 
do preserve this structure because of the following proposition. 

\begin{Prop} \label{P2.1} 
 If $v>a^*_N$ and $M\in\op{MF}_N(a)$ then a natural action 
of $\c I_{\c K}$ on $W_N(\c K_{sep})$ induces 
the $\c I_{\c K}^{(v)}$-module structure on $\wt{U}_a(M)$.
\end{Prop}

\begin{proof} All we need is just the following lemma. 
 \end{proof}

\begin{Lem} \label{L2.2} 
 If $\tau\in\c I_{\c K}^{(v)}$ with $v>a^*_N$ then 
\medskip 

{\rm a)} $\tau (u_0)/u_0\in O^*_N(\c K_{sep})$; 
\medskip 

{\rm b)} for $0\leqslant i\leqslant a$, 
$\varphi _i(\tau t^i)=1$.
\end{Lem}

\begin{proof}[Proof of Lemma]   
For $\tau\in \c I_{\c K}^{(v)}$, we have 
$\tau (u_N)=u_N+\eta _N +pw\,,$ 
where $\eta _N\in u_1^{a+1}O^+_N$ and $w\in W_{N}(\c K_{sep})$. 
For $1\leqslant i\leqslant N$, this implies  
$$\tau (u_i)=u_i+\eta _i+p^{N-i+1}w_i\,,$$
where $\eta _i\in u_1^{a+1}O_N^+$ and $w_i\in W_N(\c K_{sep})$. 
Therefore, 
$$\tau (u_1)\equiv u_1\,\op{mod}\,u_1^{a+1}O_N^+\,.$$ 
This implies part a) because  
$\tau (u_0)\equiv u_0\,\op{mod}\,u_0^{a+1}O_N^+$ 
and part b) because 
$\sigma (\tau t)/\sigma (t)\equiv 1\,\op{mod}\,u_0^aO^+_N$.     
\end{proof}

With the relation to the original problem of estimating the upper ramification numbers 
of the $\Gamma _K$-module $H$ notice now that 
$\c K=k((\tilde u))$ coincides with 
$\sigma ^{-N}\c K_0$, where $\c K_0$ is the field-of-norms 
of the $p$-cyclotomic extension $\wt{K}$ of $K$. Then for any 
$v\geqslant 0$, $\Gamma _K^{(v)}=\Gamma _K\cap \c I_{\c K}^{(v^*)}$, where 
$\varphi _{\wt{K}/K}(v^*)=v$. In particular, 
$v>N-1+a/(p-1)$ if and only if $v^*>a_N^*$. 

So, the proof of Fontaine's conjecture is reduced 
to the proof of the following theorem stated exclusively in terms of 
the field $\c K$ of characteristic $p$. 

\begin{Thm}\label{T2.3}
For any $v>a^*_N$, the group  
$\c I_{\c K}^{(v)}$ acts trivially on $\wt{U}_a(M)$. 
\end{Thm}

\section{Proof of Theorem \ref{T2.3}}
 
\subsection{Auxiliary field $\c K'$} \label{S3.1} 
Let $N^*\in\N $ and $r^*\in\Q $ be such that for $q:=p^{N^*}$, 
$r^*(q-1):=b^*\in\N $ and $v_p(b^*)=0$. 

Consider the field $\c K'=\c K(N^*,r^*)$ from \cite{Ab2}.  Remind that 
\medskip 

--- $[\c K':\c K]=q$; 
\medskip 

--- $\c K'=k((\tilde u'))$, where 
$\tilde u=\tilde u^{\prime q}E(\tilde u^{\prime\, b^*})^{-1}$ 
(here $E$ is the Artin-Hasse exponential); 
\medskip

--- the Herbrand function $\varphi _{\c K'/\c K}$ has only one edge point 
$(r^*,r^*)$. 
\medskip

For $\c K'$ and its above uniformiser $\tilde u'$ proceed as earlier to 
construct the  
lifts $O'_N(\c K')$ and $O'_N(\c K_{sep})$ obtained with respect to 
the $p$-basis $\tilde u'$. 
Introduce similarly the modules $O^{\,\prime\, 0}_N$, $O^{\,\prime\, +}_N$, the elements 
$u'_0, t'\in O_N(\c K')$ and the filtered module 
$\c S'_a$. 

\subsection{ }\label{S3.2}
Compare the old and the new lifts using their canonical 
embeddings into 
$W_N(\c K_{sep})$. 
Note that  $u_N$ is not generally an element of 
$O'_N(\c K')$ because the Teichmuller representative 
$u_N=[\tilde u]$ can't be written as a power series in $u'_N=[\tilde u']$ if $N>1$. 
However, we can easily see that for $1\leqslant i<N$, 
$u_{N-i}\in O'_N(\c K')\,\op{mod}\,p^{i+1}W_N(\c K')$. 
In particular, 
$u_1, u_0, t\in O'_N(\c K')$.

\begin{Prop} \label{P3.1}
If $\xi\in \wt{U}_a(M)$ then for any $m\in M$, $\xi (m)\in O'_N(\c K_{sep})$. 
\end{Prop} 

\begin{proof} Proceed as we proceeded at the 
end of Section \ref{S1}. Then the vectors 
$(\xi (\bar n_a), \dots ,\xi (\bar n_0))$ appear in the form 
$\bar \xi\,\op{mod}\,u_0^aO_N^+$, where $\bar \xi$ is a vector with 
coefficients in $O_N(\c K_{sep})$ such that  
\begin{equation}\label{E3.1} 
\sigma (\bar \xi )D=\bar \xi\,,
\end{equation}
and the matrix $D$ has coefficients in $O'_N(\c K')$ (use that $t\in O'_N(\c K')$). 
 We know that all coordinates 
of $\sigma ^{N-1}\bar\xi$ belong to $\sigma ^{N-1}O_N(\c K_{sep})
\subset O_N^{\prime }(\c K_{sep})$. Then \eqref{E3.1} implies step-by-step 
that the vectors 
$\sigma ^{N-2}\bar\xi , \dots ,\bar\xi $  
have coordinates in $O_N^{\prime }(\c K_{sep})$.  
 \end{proof}

\subsection{} \label{S3.3} 

Now suppose $v^*\in\Q $, 
$\c I_{\c K}^{(v)}$ acts trivially on $\wt{U}_a(M)$ for all $v>v^*$ and $v^*$ is 
the minimal with this property. The existence of $v^*$ follows from the 
left-continuity of the ramification filtration with respect to the upper numbering. 

If $v^*\leqslant a^*_N$ then our theorem is proved. 

Suppose that $v^*>a^*_N$. Choose the parameters $r^*$ and $N^*$ 
from Subsection \ref{S3.1} such that 
$a_N^*q/(q-1)<r^*<v^*$. 

For any $\alpha\in O'_N(\c K_{sep})$, set $\alpha ^{(q)}=\sigma ^{N^*}\alpha $.

\begin{Lem} \label{L3.2}
 $u_1/u_1^{\prime (q)}\equiv 1\,
\op{mod}\,u_1^{\prime\,(q)a}O_N^{\prime\,+}(\c K')$.   
\end{Lem}

\begin{proof}
Consider $b^*=r^*(q-1)\in\N $ from Subsection \ref{S3.1}. Then 
$b^*+q>q(a^*_N+1)=q(a+1)p^{N-1}$  
and 
$$u_N\equiv u^{\prime (q)}_N\,\op{mod} \left (u_1^{\,\prime\,(q)\,{a+1}}
O^+_N(\c K')+pO_N(\c K')\right )$$
This implies  
$u_1\equiv u_1^{\,\prime\,(q)}\,\op{mod}\,
u^{\,\prime\,(q)\,a+1}_1O^{\prime\,+}_N(\c K')$  
and the lemma is proved. 
\end{proof}

\begin{Cor} \label{C3.3} 
 {\rm a)} $u_0/u_0^{\,\prime\,(q)}$ is invertible in 
$O_N^{\,\prime\,0}(\c K')$;
\medskip 

{\rm b)}  $\sigma (t/t^{\,\prime\,(q)})\equiv 1
\,\op{mod}\,u_0^{\,\prime\,(q)\,a}O_N^{\,\prime\,+}(\c K')$. 
\end{Cor}

\subsection{$\c I_{\c K'}^{(v^*)}$-action}  \label{S3.4} 
Introduce the filtered module $\c S^{\,\prime\,(q)}_a$ as follows. 
\medskip 

--- $\c S^{\,\prime\,(q)}_a=O_N^{\,\prime\,0}/u_0^{\,\prime\,(q)\,a}O_N^{\,\prime\,+}$; 
\medskip 

--- for $0\leqslant i\leqslant a$, 
$\op{Fil}^i\c S_a^{\,\prime\,(q)}=t^{\,\prime\,(q)\,i}\c S_a^{\,\prime\,(q)}$;
\medskip 

--- $\varphi _i^{\,\prime\,(q)}:\op{Fil}^i\c S_a^{\,\prime\,(q)}\To\c S^{\,\prime\,(q)}_a$ 
is $\sigma $-linear such that $\varphi _i^{\,\prime\,(q)}(t^{\,\prime\,(q)\,i})=1$.  
\medskip 

Suppose $M'\in\op{MF}_N(a)$ is given similarly to $M$ by the relation 
$$(\varphi _a(\bar n_a), \dots ,\varphi _0(\bar n_0))=
(\bar n_a,\dots ,\bar n_0)\sigma ^{-N^*}C$$ 
Then we can use $\sigma ^{N^*}$ to identify 
the modules $\wt{U}_a'(M'):=\op{Hom}_{\c M\c F}(M',\c S_a')$ and 
$\wt{U}_a^{\,\prime\,(q)}(M):=\op{Hom}_{\c M\c F}(M,\c S_a^{\,\prime\,(q)})$. 
This identification is compatible with the action of the subgroups 
$\c I_{\c K'}^{(v)}$, where $v>a_N^*$. 

Note that the fields $\c K$ and $\c K'$ are isomorphic (as any 
two fields of formal power series 
with the same residue field). Choose an isomorphism $\kappa :\c K\To \c K'$ such 
that $\kappa (\tilde u)=\tilde u'$ and $\kappa |_{k}=\sigma ^{-N^*}$. 
We can extend $\kappa $ to an 
isomorphism of separable closures of $\c K$ and $\c K'$. This allows us to identify 
the groups $\c I_{\c K}$ and $\c I_{\c K'}$ and this identification is compatible 
with the appropriate ramification filtrations. Even more, we obtain an  
identification of $\wt{U}_a(M)$ with $\wt{U}_a'(M')$ and this identification respects the action 
of $\c I_{\c K}^{(v)}$ on $\wt{U}_a(M)$ and the action of 
$\c I_{\c K'}^{(v)}$ on $\wt{U}_a'(M')$ 
for any $v>a^*_N$. Therefore, $v^*$ is the maximal number such that 
$\c I_{\c K'}^{(v^*)}$ acts non-trivially on $\wt{U}_a'(M')$ and 
\medskip 

$\bullet $\ $v^*$ {\it is the maximal such that 
$\c I_{\c K'}^{(v^*)}$ acts non-trivially on $\wt{U}_a^{\,\prime\,(q)}(M)$}.

\subsection {$\c I_{\c K}^{(v^*)}$-action} \label{S3.5} 

Introduce the filtered module ${\c S}^\star _a$ as follows:
\medskip 

--- ${\c S}^{\star }_a=O_N^0\cap O'_N(\c K_{sep})/u_0^aO^+_N\cap O'_N(\c K_{sep})$; 
\medskip 

--- $\op{Fil}^i{\c S}^{\star }_a=t^i\c S_a\cap {\c S}^{\star }_a$; 
\medskip 

--- ${\varphi}^{\star } _i=\varphi _i|_{\op{Fil}^i{\c S}^{\star }_a}
:\op{Fil}^i{\c S}^{\star }_a\To {\c S}^{\star }_a$.  
\medskip 

The results from Subsection \ref{S3.2} 
allow us to identify $\wt{U}_a(M)$ with 
${U}^{\star }_a(M)=\op{Hom}_{\c M\c F}(M,{\c S}^{\star }_a)$. By 
the results from Subsection \ref{S3.3},  
there is a natural 
embedding of filtered modules ${\c S}^{\star }_a\To\c S_a^{\,\prime\,(q)}$ and, 
therefore, we can identify $\wt{U}_a(M)$ with $\wt{U}_a^{\,\prime\,(q)}(M')$. 
This identification is compatible with the action of ramification 
subgroups $\c I_{\c K}^{(v)}$ for all $v>a^*_N$. So, 
\medskip 

$\bullet $\  $v^*$ {\it is the maximal such that 
$\c I_{\c K}^{(v^*)}$ acts non-trivially on $\wt{U}_a^{\,\prime\,(q)}(M)$}.

\subsection{The end of proof of Theorem} \label{S2.5}
It remains to notice that $\c I_{\c K'}^{(v^*)}=\c I_{\c K}^{(v^*_0)}$, 
where $v^*_0=\varphi _{\c K'/\c K}(v^*)<v^*$. 

The contradiction.


\begin{thebibliography}{xxx}

\bibitem{Ab1} {\sc V.\,Abrashkin}, 
\textit{Ramification in etale cohomology}, 
{Invent. math.} (1990) {\bf 101}, 631--640

\bibitem{Ab2} {\sc V.Abrashkin}, 
\textit{Ramification filtration of the Galois group of a local field. III}, 
{Izvestiya RAN: Ser. Mat.}, {\bf 62}, no.5  (1998), 3-48;  
 English transl. 
{Izvestiya: Mathematics} {\bf 62}, no.5, 857--900

\bibitem{De} 
{\sc P.Deligne } 
\textit{Les corps locaux de caract\'eristique $p$, limites
 de corps locaux de caract\'eristique $0$}, 
{Representations of reductive groups over a local field, Travaux en cours},  
Hermann, Paris, 1973, 119-157 



\bibitem{Fo} {\sc J.-M.\,Fontaine}, 
\textit{Il n'y a pas de vari\'et\'e abelienne sur $\Z $}, 
{Invent. math.} (1990) {\bf 101}, 631--640

\bibitem{FL} {\sc J.-M.\,Fontaine, G. Laffaille},  \textit{Construction de 
repr\'esentations $p$-adiques}, {Ann. Sci. \'Ecole Norm. Sup., 4 Ser.} 
{\bf15} (1982), 547--608



\bibitem{Se1} {\sc J.-P.Serre}, \textit {Local Fields} 
Berlin, New York: Springer-Verlag, 1980 

\bibitem{Wtb1} {\sc J.-P. Wintenberger}, 
\textit{Le corps des normes de certaines extensions
infinies des corps locaux; application} 
{Ann. Sci. Ec. Norm. Super.,
IV. Ser}, {\bf16} (1983), 59--89


\bibitem{Wtb2} {\sc J.-P.Wintenberger}, \textit {Extensions de 
Lie et groupes d'automorphismes des 
corps locaux de caractéristique p.} (French) 
{C. R. Acad. Sci. Paris} S\'er. {\bf A-B 288} (1979), 
no. 9, A477–-A479




\end{thebibliography}
\end{document}